\documentclass[11pt, twoside]{amsart}
\usepackage{epsfig, amssymb,  amsfonts,xspace,amsthm,amscd}
\usepackage{mathtools}
\usepackage[latin1]{inputenc}
\usepackage[T1]{fontenc}
\usepackage[all]{xy}
\usepackage{hyperref}
\numberwithin{equation}{section}

\newtheorem{lemma}{Lemma}[section]
\newtheorem{proposition}[lemma]{Proposition}
\newtheorem{theorem}[lemma]{Theorem}

\newtheorem{corollary}[lemma]{Corollary}
\newtheorem{rem}[lemma]{Remark}

\newtheorem*{special theorem}{My Specially-Named Theorem}

\newcommand{\comment}[1]{}

\begin{document}


\pagestyle{plain}


\title{On cohomologically Trivially modules over finite $p$-groups }

\author{Yassine Guerboussa$^{\dag}$ and Maria Guedri$^{*}$}

\address[$^{\dag}$]{Department of Computer Sciences, University Kasdi Merbah Ouargla, Ouargla, Algeria \\ {\tt Email: yassine\_guer@hotmail.fr}}

\address[$^{*}$]{Department of Mathematics, Fr\`{e}res Mentouri University-Constantine 1, Ain El Bey road, 25017 Constantine. Algeria
	\\ {\tt Email: guedri\_maria@yahoo.fr}}

\date{\today}

\maketitle

\begin{abstract} We show that every finitely generated cohomologically trivial module over $RG$, where $G$ is a finite $p$-group and $R$ is a $p$-adic ring, splits as the direct sum of a finite cohomologically trivial $RG$-module and a free $RG$-module.  Along the way, we also establish other results concerning generators and relators of such modules.

\end{abstract}\vspace{1cm}

\footnotesize
Keywords: finite $p$-groups, group cohomology, zeta functions.


\normalsize

 \section{Introduction}
The purpose of this paper  is to note  some general properties  of \textit{cohomologically trivial }(\textit{CT} for brevity) modules over finite $p$-groups.  The interest in such modules---though  originated in class field theory (see e.g., Nakayama \cite{Nakayama1})---has been primarily motivated in recent years by questions regarding finite $p$-groups and their automorphisms; see \cite{AliPowerfulCT, AliCT, Benmoussa, AMY} and  \cite{Mexican,YasMar1}.

Throughout we let $p$ be a (rational) prime,  $G$ be  a finite $p$-group, and $R$ be  the ring of integers of a $p$-adic field $K$ (i.e., $K$ is a finite extension of the field  of $p$-adic numbers $\mathbb{Q}_p$).  For every $RG$-module $A$ we can consider the torsion $R$-submodule  
$$T(A)=\{a\in A\mid \pi^la=0 \mbox{ for some integer } l \geq 0 \},$$
where $\pi$ is a uniformizing element of $R$.  It is readily seen that
$T=T(A)$ is invariant under the action of $G$,  that is, $A$ is an $RG$-submodule  of $A$. 

If $A$ is finitely generated over $RG$ (or equivalently over $R$), then, since $R$ is a principal ideal domain, 
 $T=T(A)$ can be written uniquely as a direct sum of finitely many $R$-modules of the form $R/(\pi^{n_i})$, with $n_i\in \mathbb{N}$; in particular,   $T$ is finite  of  $p$-power  order (we refer to \cite[Theorem 1.2]{BG} for the  relationship between the invariants $n_i$ and the invariants of $T$ as an abelian $p$-groups).   Moreover,
 $A/T$ is  free over $R$ and we have  
$$A=T\oplus A/T \quad (\mbox{over } R).$$

One of our aims here is show that the above decomposition holds in fact over $RG$  provided $A$ is CT.   Before proceeding further, we recall that
a module $A$ over  $G$ (that is, $A$ is a $\mathbb{Z}G$-module)  is called \textit{CT} if the  (Tate)  cohomology groups  $\hat{H}^n(S,A)$ are zero for all subgroups $S\leq G$ and all integers $n$.   

\begin{theorem}\label{Theorem A}
	Let $A$ be a  finitely generated $RG$-module. If $A$ is CT, then so are  $T$ and $A/T$; moreover, $A/T$ is free over $RG$  and  we have
	$$A= T\oplus A/T \quad (\mbox{over } RG).$$
\end{theorem}
The assumption that $A$ is CT  in this theorem is essential. Indeed, take for example   $A=RG/\mathfrak{I}^n$, where  $\mathfrak{I}$ is the augmentation ideal and $n$ is an integer $\geq 2$.  As is well known, $\mathfrak{I}/\mathfrak{I}^2\cong G^{ab}$ is finite; by using the obvious epimorphism $\mathfrak{I}/\mathfrak{I}^2\otimes_R\mathfrak{I}^{j-1}/\mathfrak{I}^{j}\to \mathfrak{I}^{j}/\mathfrak{I}^{j+1}$, we see via induction on $j$ that $\mathfrak{I}^{j}/\mathfrak{I}^{j+1}$ is finite for all integers $j\geq 1$; in particular $\mathfrak{I}/\mathfrak{I}^n$ is finite.  Also, by observing that $RG/\mathfrak{I}\cong R$, we see that the torsion part of $A$ is $T=\mathfrak{I}/\mathfrak{I}^n$. But, $A$ as an $RG$-module doesn't split over $T$ (observe that $A$ has rank $1$; see below). Note that $\hat{H}^0(G,A)= \mathfrak{I}^{n-1}/(R\cdot  N+\mathfrak{I}^{n})$ is non-zero, where $N$ denotes the norm $\sum_{x\in G}x$ (observe here  that $\mathfrak{I}\cap R\cdot N=0$).

Theorem \ref{Theorem A} shows that the intriguing part of the category of  finitely generated CT $RG$-modules  lies in the subcategory of those which are finite;  of particular interest are the latter which are  indecomposable  (i.e., the building blocks of the former according to the Krull-Schmidt theorem).  A deeper knowledge of the latter will be of great importance  for understanding the counter-examples to Schmid's conjecture (\cite[p. 3]{Schmid}): \textit{For every non-abelian $p$-group $G$, $Z(\Phi(G))$ is not CT over
$G/\Phi(G)$, where  $\Phi(G)$  denotes the Frattini subgroup of $G$.}  It is worth noting that, so far, the only known counter-examples are the $2$-groups of order $2^8$ with  {\sf IdSmallGroup} $298,\dots,307$  in  GAP Library \cite{GAP}, which were found by Abdollahi  \cite{AliCT} (see also \cite{AMY} and  \cite{YasMar1}).   Note that Schmid's conjecture is closely related to the  longstanding problem  of whether 
every non-abelian finite $p$-group $G$ has a non-inner automorphism of order $p$; we refer to \cite{AliPowerfulCT, Benmoussa} for  details and to the more recent papers \cite{kom2024, MaNoninnerConjecture} for an updated bibliography.

Let  $\overline{K}$ denote the residual field $R/(\pi)$.   To every  finitely generated module $RG$-module $A$ we may attach the following two invariants:
\begin{equation}\label{Basic ranks over Z_p}
d_{R}(A)=\dim_{\overline{K}} A/\pi A   \quad \mbox{and } \quad  r_{R}(A)=\dim_{\overline{K}} A/(\pi A+[A,G]).
\end{equation}
Note that $d_{R}(A)$ is equal to the number of cyclic factors in the canonical decomposition of $A$ as an $R$-module (equivalently, $d_{R}(A)$ is equal to the size of an arbitrary generating set of $A$ over $R$; see Lemma \ref{Lemma elementary rank} below).     Similarly,   $r_{R}(A)$ is equal to the size of any minimal generating set of $A$ over $RG$ (this  depends heavily on the fact that $RG$ is a local ring; see Lemmas \ref{The base ring is local} and \ref{Lemma elementary rank}).  Accordingly, we may call $r_{R}(A)$  the \textit{$RG$-rank}  of $A$. 

If we are given   a short exact sequence of  $RG$-modules $M\rightarrowtail L\twoheadrightarrow A$,  with $L$ free of finite rank, then it is readily seen that
\begin{equation}
d_{R}(M)=d_{R}(L)-d_{K}(A),
\end{equation}
where $d_K(A)=\dim_K A\otimes_R K$ (which we may call the \textit{torsion-free} rank of $A$).   For the $RG$-rank of $M$ we have the following.
\begin{theorem}\label{Number of relations for the ZpG-module A}  Let $A$ be a finitely generated $RG$-module, and let  $M\rightarrowtail L\twoheadrightarrow A$ be an exact sequence 
of $RG$-modules,  with $L$ free of finite rank.	Then we have
	$$r_{R}(M)=r_{R}(L)-d_K(A_G)+ d_{R}(H_1(G,A)),$$
	where   $A_G=A/[A,G]$.	
\end{theorem}
Interestingly, the latter   implies that
the cohomological type of a finite  $RG$-module $A$ is determined by the minimal number of its defining relations  (encoded in $r_R(M)$). More precisely:
\begin{corollary}\label{CT modules are characterized by their relations } For a finite $RG$-module $A$, 
	the following two statements are equivalent:
	\begin{itemize}
		\item[(i)]  $A$ is  CT.
		\item[(ii)] For all $RG$-modules $L$ and $M$  as in Theorem \ref{Number of relations for the ZpG-module A}, we have
		$r_R(M)=r_R(L)$.
	\end{itemize} 
\end{corollary}

Note that, for every free $RG$-module $L$ of finite rank, several asymptotic properties of the finite quotients of $L$ that are CT are reflected in the sequence $n\mapsto c_n(L)$, where  $c_n(L)$ denotes the number of free submodules of $L$ of index $q^n$, with $q=\mathrm{card}(\overline{K})$. We may then encode the previous sequence  in the (formal) Dirichlet series
$$
Z_{\text{coh}}(L;s)= \sum_{n\geq 0}c_nq^{-ns}=\sum_{M \text{ free}}\vert L:M\vert^{-s}
$$
The above is a particular case of the  zeta functions   treated in Solomon \cite{Solomon}. The latter shows in particular---by using combinatorial methods---that  $Z_{\text{coh}}(L;s)$
 is a rational function in $q$.   Another proof of this, based on Tate's thesis,    can be found in Bushnell and Reiner \cite{Bushnell and Reiner}.  As pointed out in \cite{YasMar1},  the above series can be viewed as a $p$-adic integral (à la Igusa \cite{Iguza}), and the basic theory of such integrals shows that the former is indeed a rational function in $q$.  The advantage of working with $p$-adic integral lies in their deep connections with algebraic geometry.  The latter 
  provides  a way for 'reducing' modulo $\pi$, which  allows us to work with varieties over finite fields and their  zeta functions to obtain more refined results, e.g., functional equation, uniformity  (cf. Igusa \cite[Chap. 9]{Iguza} and Voll \cite{Voll}).  It seems, however, that when dealing with
$Z(L;s)$ we   encounter  the worst case in which the reduction is bad in the sense of Denef (Igusa, \textit{loc. cit}).

The rest of the paper is divided into three sections.  In Section 2, we discuss the basic facts concerning the invariants (\ref{Basic ranks over Z_p}), and  we prove Theorem \ref{Number of relations for the ZpG-module A}.  The results in Section 3 are preliminaries for the proof of Theorem \ref{Theorem A}; except some slight generalizations, these are basically known in the literature, but we supply  the core of the proofs (whenever possible) to help the reader in following the sequel.
Note that Corollary \ref*{CT modules are characterized by their relations } follows immediately from Theorem  \ref{Number of relations for the ZpG-module A} and  Theorem \ref{Gaschutz-Uchida} (due to Gasch\"{u}tz and Uchida). The last section is completely devoted to the proof of Theorem \ref{Theorem A}.

 \section{Generators and presentations of $RG$-modules}\label{Presentations of OG modules}
  As is well known, for every field (or more generally a division ring) $k$ of characteristic $p$, the  group ring $kG$ is local   whose  maximal  ideal is equal to the augmentation ideal  (see e.g., \cite[Chap. IX, \S 1]{SerreLF}). This result can be easily  extended  to group rings of the form $SG$, where $S$ is a local ring  with residual division ring of characteristic $p$. 
 
 \begin{lemma}\label{The base ring is local}
 	If  $S$ is a local ring with maximal ideal $\mathfrak{q}$ such that
 	$S/\mathfrak{q}$ has characteristic $p$, then $SG$ is a local ring whose  maximal ideal is given by $\mathfrak{m}=\mathfrak{I}+\mathfrak{q}\cdot 1$ (where $\mathfrak{I}$ denotes the augmentation ideal of $SG$), and we have $SG/\mathfrak{m}\cong S/\mathfrak{q}$.   	
 \end{lemma} 
We record a proof for the reader's convenience.
 \begin{proof}
 	Set $k=S/\mathfrak{q}$; we have a ring morphism $\bar{}: SG\to  kG$ defined  by reducing the coefficients  modulo $\mathfrak{q}$.   Define 
 	$\mathfrak{m}$ to be the inverse image
of
 the  augmentation ideal of $kG$ by the previous morphism.  It follows immediately that  $\mathfrak{m}=\mathfrak{I}+\mathfrak{q}\cdot 1$ and  that $SG/\mathfrak{m}\cong kG/\bar{\mathfrak{m}}\cong k$; in particular,  $\mathfrak{m}$ is a  maximal (left) ideal of $SG$. If  $SG$ has another  maximal left ideal  $\mathfrak{l}$, then  $\bar{\mathfrak{m}}+\bar{\mathfrak{l}}=kG$; but, since $\bar{\mathfrak{m}}$  is equal to  the Jacobson radical of $kG$, we see by Nakayama's lemma  that $\bar{\mathfrak{l}}=kG$, that is,
 	$\mathfrak{l}=SG+\mathfrak{q}\cdot SG$.  Applying Nakayama's lemma once again (to $SG$ viewed as an $S$-module), we obtain the contradiction
 	$\mathfrak{l}=SG$, and  the lemma follows.
 \end{proof}

 For $S=R$ (the ring of integers of the $p$-adic field $K$),   the maximal ideal of  $RG$ is given by 	$ \mathfrak{m}=\mathfrak{I}+(\pi)\cdot 1$. Hence, for every $RG$-module $A$ we have
 $$\mathfrak{m}A=[A,G]+\pi A.$$
 Obviously, we may view $A/\mathfrak{m}A$ as a $\overline{K}$-vector space, and then consider its dimension, which gives the invariant $r_R(A)$  already defined in (\ref{Basic ranks over Z_p}).
   
 \begin{lemma}\label{Lemma elementary rank}
If $A$ is a finitely generated $RG$-module,  then every 
minimal generating set of $A$ over $RG$ has size  $r_{R}(A)$.  Furthermore, we have $r_{R}(A)=d_{R}(A_G)$.
 \end{lemma}
\begin{proof}
The first statement follows by 
Nakayama's lemma; the second follows by   observing that $\pi A_G=\mathfrak{m}A/[A,G]$ (which shows that $A_G/\pi A_G\cong A/ \mathfrak{m}A$).
\end{proof}

\begin{proof}[Proof of Theorem \ref{Number of relations for the ZpG-module A}]
We have $H_1(G,L)=0$ (since $L$ is free over  $RG$);   the  long exact sequence of homology associated  to $M\rightarrowtail L\twoheadrightarrow A$ implies, therefore, the following:
	$$
	\xymatrix{ 0 \ar[r] & H_1(G,A) \ar[r] & M_G \ar[r] & L_G
		\ar[r] &	A_G\ar[r] & 0}.
	$$
	By identifying $H_1(G,A)$ with an $R$-submodule $N$ of $M_G$, and setting $M'=M_G/N$, the above reduces to
	$$E:
	\xymatrix{ 0 \ar[r] &  M'\ar[r] & L_G
		\ar[r] &	A_G\ar[r] & 0}.
	$$
Observe that	multiplication by $\pi$  induces a morphism $\pi:E\to E$;  using the  ker-coker lemma  we obtain the exact sequence
	$$
	\xymatrix{ 0 \ar[r] & _{\pi}M'  \ar[r] &  _{\pi}(L_G) \ar[r]    & _{\pi}(A_G)  \ar[ld]  & \\
		&	& M'/\pi M'  \ar[r]  & L_G/\pi L_G  \ar[r] & A_G/\pi A_G  \ar[r] &0\, 
	}
	$$ 
The augmentation map $RG\twoheadrightarrow R$ allows us to  identify $L_G$ with  $R^{r}$, where $r=d_{R}(L_G)$. This shows in particular that  $L_G$ is  torsion-free over $R$, that is, $_{\pi}(L_G)=0$.  It follows by the  latter exact sequence that $_{\pi}M'=0$ and that the following sequence is exact:
	\begin{equation}\label{Exact sequence (s)}
	\xymatrix{ 
		0\ar[r] 	& _{\pi}(A_G)  \ar[r] 	& M'/\pi M'  \ar[r]  & L_G/\pi L_G  \ar[r] & A_G/\pi A_G  \ar[r] &0  }
	\end{equation}
	From $_{\pi}M'=0$ we obtain that $M'=M_G/N$ is a torsion-free $R$-module, so $M'$ is  free over $R$. Subsequently we have $M_G\cong N\oplus M'$ and $d_{R}(M_G)=d_{R}(N)+d_{R}(M')$, or, equivalently,  (cf. Lemma \ref{Lemma elementary rank})
	 $$r_{R}(M)=d_{R}(H_1(G,A))+d_{R}(M').$$ 
	 	On the other hand, the  exact sequence (\ref{Exact sequence (s)}) implies  that 
	$$d_{R}(M')=\dim_{\overline{K}}   (_{\pi}(A_G)) +d_{R}(L_G)-d_{R}(A_G).$$
	Thus
	$$ r_{R}(M)=d_{R}(L_G)-d_{R}(A_G)+\dim_{\overline{K}} (_{\pi}(A_G))+d_{R}(H_1(G,A)).$$
	To conclude the proof,  observe that $\dim_{\overline{K}} (_{\pi}(A_G))$ is equal to the $R$-rank of the torsion part of $A_G$, and hence $d_{R}(A_G)-\dim_{\overline{K}} (_{\pi}(A_G))=d_K(A_G)$ (observe also that we have $d_{R}(L_G)=r_{R}(L)$ by Lemma \ref{Lemma elementary rank}).
\end{proof}

\section{Criteria for a $G$-module to be CT}
We start by giving a slight generalization of Nakayama \cite[Prop. 9]{Nakayama2}; cf.    Serre \cite[Lemma 5, Chap. IX]{SerreLF}.  The argument we shall give  is 	essentially the same as that given  in Serre, \textit{loc. cit.}
\begin{lemma}\label{Modules annihilated by pi are free over residual fields}
Let $A$ be an $RG$-module (not necessarily finitely generated) such that $\pi A=0$.  Then the following conditions are equivalent:
\begin{itemize}
	\item[(i)] $\hat{H}^{n}(G,A)=0$ for some integer $n$.
	\item[(ii)] $A$ is a free $\overline{K}G$-module.
	\item[(iii)] $A$ is CT over $G$. 
\end{itemize}
\end{lemma}
\begin{proof}
Clearly, we need only show that (i) implies (ii). Suppose (i) is true; 
by dimension shifting we may take $n=-1$,  so we have $H_1(G,A)=0$.  Pick a family $\{a_i\}$ in $A$  in such a way that  $\{\overline{a_i}\}$ forms a $\overline{K}$-basis of $A_G=A/\bar{\mathfrak{I}}A$ (where $\bar{\mathfrak{I}}$ is the augmentation ideal  of $\overline{K}G$).  Since  $\bar{\mathfrak{I}}$ is nilpotent, we have for every $\overline{K}G$-module $B$: 
\begin{equation}\label{Analogue of Nakayama}
\bar{\mathfrak{I}}B=B \implies B=0.
\end{equation}
(In other words, we have an analogue of Nakayama's lemma that works without the assumption that $B$ is finitely generated.) 
Applying the latter for $B=A/A'$, where $A'$ is  the submodule generated by $\{a_i\}$, we obtain $A=A'$, that is, $\{a_i\}$ generates $A$.  Now, if $L$ is a free $\overline{K}G$-module on a family of indeterminates $\{x_i\}$, the map $x_i\mapsto a_i$ induces an epimorphism $s:L \to A$; so we have an exact sequence $M\rightarrowtail L\stackrel{s}{\twoheadrightarrow} A$, where $M=\ker s$. Passing to homology we obtain the exact sequence
$$\xymatrix{ 0= H_1(G,A) \ar[r] & M_G \ar[r] & L_G	\ar[r]^{\bar{s}} &	A_G\ar[r] & 0}.
$$ 
Our assumption on $\{\overline{a_i}\}$ ensures that  $\bar{s}$ is an isomorphism, so we have $M_G=0$, and hence  $M=0$ by   (\ref{Analogue of Nakayama}).   It follows that $L\stackrel{s}{\to} A$ is an isomorphism, which shows that  $A$ is free over $\overline{K}G$, as desired.
\end{proof}

The above result can be used in a more general setting as follows.  If $M$ is an $RG$-module which is torsion-free over $R$, then we have the obvious exact sequence
\begin{equation}\label{Short exact sequence of multiplication by pi}
\xymatrix{    M  \,\ar@{>->}[r]^{\pi} & M
	\ar@{->>}[r] & M/\pi M}.
\end{equation}
Applying $\hat{H}^{*}(G,-)$ to the latter, we find  that if $\hat{H}^{*}(G,M)$ vanishes in two consecutive dimensions, say $n-1$ and $n$, then $\hat{H}^{n}(G,M/\pi M)=0$, and hence
$M/\pi M$ is  free over $\overline{K}G$ by the previous lemma.  Note also that
if the latter holds, then   $M/\pi M$ is free over  $\overline{K}S$ for every $S\leq G$; so when applying
 $\hat{H}^{*}(S,-)$ to (\ref{Short exact sequence of multiplication by pi}) we find that $\hat{H}^{n}(S,M)\stackrel{\pi^*}{\to} \hat{H}^{n}(S,M)$ is an isomorphism for all $n$; but, since multiplication by $\pi$ is nilpotent, it follows that   $\hat{H}^{n}(S,M)=0$ for all $n$, which shows that  $M$ is CT over $G$. To sum up, we have:
\begin{rem}\label{Twin integers in the torsion-free case}
For every $RG$-module 	$M$ which is torsion-free over $R$, the following conditions are equivalent:
\begin{itemize}
	\item[(i)] $\hat{H}^{n}(G,M)$ for two consecutive integers $n$.
	\item[(ii)] $M/\pi M$ is a free $\overline{K}G$-module.
	\item[(iii)] $M$ is CT over $G$. 
\end{itemize}

\end{rem}

The following result is a slight generalization of  Rim \cite[Prop. 3.6]{Rim 1959} (cf. Serre \cite[Thm. 5, Chap. IX]{SerreLF}).
\begin{proposition}\label{Presentations of CT modules are free}
If $A$ is an $RG$-module which is  $R$-projective, then we have 
	$$A \mbox{ is CT }\iff  A \mbox{ is  free over } RG.$$
\end{proposition}
\begin{proof}
We need only prove the direct implication ($\Rightarrow$).  Suppose  $A$ is CT and choose a presentation $M\rightarrowtail L\stackrel{s}{\twoheadrightarrow} A$ of $A$ over $RG$, with $L$ free. Since $A$ is $R$-projective, the functor $\mathrm{Hom}_{R}(A,-)$ is exact; applying this to the preceding presentation and passing  to (ordinary) cohomology, we obtain the exact sequence
	$$
\xymatrix{ \mathrm{Hom}_{RG}(A,L)\ar[r]^{s_{*}} & \mathrm{Hom}_{RG}(A,A)
	\ar[r] & H^1(G, X)	}
$$
where $X=\mathrm{Hom}_{R}(A,M)$.  

We claim that $H^1(G, X)=0$.  By applying $\mathrm{Hom}_{R}(A,-)$ to  $M\stackrel{\pi}{\rightarrowtail}M\twoheadrightarrow M/\pi M$, we obtain the exact sequence
$$
\xymatrix{ 0 \ar[r] & X \ar[r]^{\pi}& X \ar[r]& \mathrm{Hom}_{R}(A/\pi A,M/pM)
	\ar[r] & 0},
$$
hence we may identify   $X/\pi X$  to  $X'=\mathrm{Hom}_{\overline{K}}(A/\pi A,M/pM)$.   We know by 
Remark
\ref{Twin integers in the torsion-free case} that $A/\pi A$ is a free $\overline{K}G$-module,  so $X'$ is coinduced and hence CT.   This shows that $X/\pi X$ is CT and our claim follows by another application of  Remark
\ref{Twin integers in the torsion-free case}.  

Now it follows that 	$
 \mathrm{Hom}_{RG}(A,L) \stackrel{s_{*}}{\to} \mathrm{Hom}_{RG}(A,A)
$ is surjective; so there exists  $u\in \mathrm{Hom}_{G}(A,L)$  such that  $s\circ u=1_A$ and hence $A$ is a direct factor of $L$. This proves that  $A$ is  projective over $RG$, so it is free because  $RG$ is local (cf. Lemma \ref{The base ring is local}.
\end{proof}

Suppose now that $A$ is an arbitrary $RG$-module.  Recall that for every normal subgroup $N\lhd G$ we have the inflation-restriction exact sequence (see e.g., \cite[Prop. 4, Chap. VII ]{SerreLF}):
\begin{equation}\label{Inflation-restriction exact sequence}
\xymatrix{ 0 \ar[r] & \hat{H}^1(G/N,A^N) 	\ar[r]^{\text{Inf}} & \hat{H}^1(G,A)
	\ar[r]^{\text{Res}} & \hat{H}^1(N,A)}.
\end{equation}
The latter can be pushed to a higher dimension $n>1$ provided that $\hat{H}^i(N,A)$ is zero for  $i=1,\ldots, n-1$ (cf. \cite[Prop. 5, Chap. VII]{SerreLF}).

\begin{lemma}\label{Serre's on inflation map}
If $A$ is an arbitrary $RG$-module and  $N$ is a normal subgroup of $G$ such that $\hat{H}^i(N,A)$ is zero for  $i=1,\ldots, n-1$, then the following natural sequence is exact:
$$
\xymatrix{ 0 \ar[r] & \hat{H}^n(G/N,A^N) 	\ar[r]^{\text{Inf}} & \hat{H}^n(G,A)
	\ar[r]^{\text{Res}} & \hat{H}^n(N,A)}.
$$
\end{lemma}
(A proof can be found in \cite{SerreLF}, \textit{loc. cit.})

\medskip

Note that Remark
\ref{Twin integers in the torsion-free case} can be used to deal with an arbitrary $RG$-module $A$ by considering any presentation $M\rightarrowtail L\twoheadrightarrow A$, with $L$ free over $RG$ (subsequently $M$ is $R$-free), and then applying $\hat{H}^*(G,-)$ to see that $\hat{H}^n(G,A)\cong \hat{H}^{n+1}(G,M)$ for all integers $n\geq 1$.  This yields the following well-known result of Nakayama \cite[Theorem 1]{Nakayama1} (cf. \cite[Thm. 5, Chap. IV]{SerreLF}).

\begin{theorem}\textnormal{(Nakayama.)}\label{Nakayama twin integers}
	Let $A$ be an arbitrary $RG$-module (or just $G$-module). Then $A$ is CT over $G$ if and only if $\hat{H}^n(G,A)=0$ for two consecutive integers $n$.
\end{theorem}

Less known in the literature is the following result, established independently by Gasch\"{u}tz \cite{Gasch1965} and Uchida \cite{Uchida}, which improves Nakayama's theorem 
for $A$ finite.

\begin{theorem}\textnormal{ (Gasch\"{u}tz-Uchida)}\label{Gaschutz-Uchida}
	Let $A$ be a finite $G$-module.  Then $A$ is CT over $G$ if and only if $\hat{H}^n(G,A)=0$ for some integer $n$.
\end{theorem}
(It is worth noting that the above follows easily by induction on the order of $G$ from 
 the \textit{five-term exact sequence} (a completion of \ref{Inflation-restriction exact sequence} by two other terms); for the latter see, e.g., Weibel \cite[\S 6.8, p. 196]{Weibel}.)

\section{A CT module splits over its torsion part}\label{The decomposition as torsion + torsion free} 
This section is devoted to proving Theorem \ref{Theorem A}.  Throughout we assume that $A$ is a  finitely generated $RG$-module and that it is  CT in addition.  Recall that $T=T(A)$ denotes the torsion part of $A$ over $R$. We need some auxiliary results.

\begin{lemma}\label{Enough to get T CT}  Let $A$ be a  finitely generated  $RG$-module which is CT. Then the following statements are equivalent:
	\begin{itemize}
		\item[(i)] $T$ is CT.
		\item[(ii)] $A/T$ is CT.
		\item[(iii)] $A= T\oplus A/T$ over $RG$.
	\end{itemize}

\end{lemma}
\begin{proof}
By applying  $\hat{H}^*(G,-)$ to the exact sequence
 $T\rightarrowtail A\twoheadrightarrow A/T$ we see immediately that, for all integers $n$,
$$
\hat{H}^n(G,A/T)\cong \hat{H}^{n+1}(G,T).
$$
It follows in particular that (i) is equivalent to (ii).  Suppose now that  (ii) holds.  By  Prop. \ref*{Presentations of CT modules are free},  $A/T$ is a free  $RG$-module, so the obvious $RG$-morphism $A\twoheadrightarrow A/T$ has a section, and  (iii) follows. The implication '(iii) $\Rightarrow$ (i)' is immediate   by the additivity of  $\hat{H}^*(G,-)$.
\end{proof}
 
The next result  reduces the  proof of  Theorem \ref{Theorem A} to the case where $G$  has  order $p$. 

\begin{lemma}\label{Reduction to G of order p}
If Theorem \ref*{Theorem A} holds for $G$ cyclic of order $p$, then it holds every for  finite $p$-group $G$.
\end{lemma}
 \begin{proof}
 We proceed by induction on the order of $G$.  Suppose that $\vert G\vert>p$, and pick a minimal normal subgroup $N\leq G$ (that is, $N$ is a  central subgroup of order $p$).  Since $A$ is CT,  $\hat{H}^n(N,A)$ vanishes for all integers $n$, so  by  Lemma \ref{Serre's on inflation map}  we have, for all $n\geq 1$,
 $$H^n(G/N,A^N)\cong H^n(G,A)=0,$$
Now, by Nakayama's theorem (Theorem \ref{Nakayama twin integers}), $A^N$ is CT over $G/N$,  and so, by the induction hypothesis,  $T(A^N)$ is CT over $G/N$.  On the other hand, as $N$ has order $p$, the assumption in the lemma ensures that 
 $T=T(A)$ is CT over $N$; applying then  $H^*(N,-)$ to 
 $T\rightarrowtail A\twoheadrightarrow A/T$ gives the exact sequence
 $$
 \xymatrix{ 0 \ar[r] &  T^N\ar[r] & A^N
 	\ar[r] & (A/T)^N	\ar[r] & H^1(N,T)=0}.
 $$
It follows that $A^N/T^N$ has no $\pi$-torsion, that is, $T(A^N)=T^N$, and so       $T^N$
 is CT over $G/N$.  Next, by using the inflation-restriction exact sequence 
	$$
	\xymatrix{ 0 \ar[r] & \hat{H}^1(G/N,T^N) 	\ar[r]^{\text{Inf}} & \hat{H}^1(G,T)
	\ar[r]^{\text{Res}} & \hat{H}^1(N,T)},
$$
we obtain $\hat{H}^1(G,T)=0$; it follows by Gasch\"{u}tz-Uchida (Theorem \ref{Gaschutz-Uchida}) that $T$ is CT over $G$, and Lemma \ref{Enough to get T CT} concludes the proof.
 \end{proof}

Now we need to deal with modules $V$ over $\overline{K}G$ for $G=\langle g\rangle$  cyclic of order $p$.  Plainly, every such a  $V$ can be viewed as a module over the polynomial ring $\overline{K}[X]$, where $X$ acts on $V$ as $g-1$ and $X^pV=0$ (which follows  because $(g-1)^p=g^p-1=0$).  The  point here is that  $\overline{K}[X]$ is a principal ideal domain, and so every finitely-generated (equivalently, finite-dimensional over $\overline{K}$) $\overline{K}G$-module $V$ can be uniquely expressed as a direct sum of cyclic modules of the form $V_n=\overline{K}[X]/(X^r)$, with $1\leq  r\leq p$. 

\begin{lemma}\label{Modules over F_p[X]}
Suppose that $G$ is cyclic of order $p$, and let  $V$ be a finite-dimensional  $\overline{K}G$-module.  If $\mathrm{Hom}_{\overline{K}}(V,V)$ is free over $\overline{K}G$, then so is $V$.
\end{lemma}

\begin{proof}
Firstly, note that $\mathrm{Hom}_{\overline{K}}(V,V)$ can be identified to $V^*\otimes_{\overline{K}} V$ (clearly, both are viewed as $G$-modules under  the obvious diagonal actions) via the $\overline{K}G$-isomorphism 
$$\alpha \otimes v \stackrel{\varphi}{\longmapsto}
  \varphi_{\alpha \otimes v}:u\mapsto \alpha(u)v,$$
for $\alpha \in  V^*$ (the $\overline{K}$-dual of $V$) and $u,v\in V$.    (Since both of  $\mathrm{Hom}_{\overline{K}}(V,V)$ and  $V^*\otimes_{\overline{K}} V$ have dimension $(\dim_{\overline{K}}V)^2$ over $\overline{K}$, it suffices to check that $\varphi$ is surjective; the latter follows by observing  that every $f \in \mathrm{Hom}_{\overline{K}}(V,V)$ can be expressed as $f=\sum_i \varphi_{v_i^* \otimes f(v_i)}$, where $\{v_i\}$ is a fixed  $\overline{K}$-basis of $V$.) 

Now the assumption in the lemma means that $V^* \otimes_{\overline{K}} V$  is a free  $\overline{K}G$-module; in other words,   all the indecomposable factors of the latter are isomorphic to $KG$.  Let $V=\bigoplus_{i} V_{r_i}$  be the canonical decomposition of $V$ over $\overline{K}G$  (so $ V_{r_i}\cong  \overline{K}[X]/(X^{r_i})$, with $1\leq r_i\leq p$).  Clearly, we have $V^*= \bigoplus_{i} V_{r_i}^*$ and, for every index $i$,  $V_{r_i}^*\cong V_{r_i}$.  It follows that 
$$V^* \otimes_{\overline{K}} V=\bigoplus_{i,j}  V_{r_i} \otimes_{\overline{K}} V_{r_j}.$$
By the uniqueness of the decomposition, every indecomposable direct summand of $V_{r_i} \otimes_{\overline{K}} V_{r_j}$ is isomorphic to $\overline{K}G$ and hence  has dimension $p$ over $\overline{K}$.    It follows that $p$ divides $\dim_{\overline{K}} V_{r_i} \otimes_{\overline{K}} V_{r_j}$, that is, $p$ divides $r_ir_j$ for every pair  of indexes $(i,j)$.   If $V$ wasn't free, then  one of the invariants $r_i$ would be strictly less than $p$, and hence $r_i^2$ is prime to $p$, which contradicts the preceding fact.  The lemma follows. 
\end{proof}

\textit{Remark.}  The study of the invariants (i.e., the dimensions of the indecomposable factors) of  $V_{r} \otimes_{\overline{K}} V_{s} $   is an interesting combinatorial problem. Clearly, this  problem also arises in the more general situation where $G$ is  cyclic  of order $p^n$ (for every $n\geq 1$);  the associated indecomposable modules here are indexed by the  integers $r=1,\ldots,p^n$.   We refer  to Green \cite{Green} and to Glasby et al. \cite{Glasby et al.} (and the references therein) for  further discussions.

\begin{proof}[Proof of Theorem \ref{Theorem A}]
By Lemma \ref{Enough to get T CT}, we need only show that $B=A/T$ is CT, and, by Lemma \ref{Reduction to G of order p}, it suffices to show the latter for $G$  cyclic of order $p$. Since $B$ is torsion-free, we can argue as in the proof of  Prop. \ref{Presentations of CT modules are free} to see that $X=\mathrm{Hom}_{R}(A,B)$ is CT.  Observe that every $R$-morphism $A\to B$ vanishes on $T$ (because $B$ is torsion-free); so   we may identify $X$ to $\mathrm{Hom}_{R}(B,B)$.  The assumption that $A$ is finitely generated implies that $B$ is $R$-free.  Hence, as in the proof of Prop. \ref{Presentations of CT modules are free},  by applying $\mathrm{Hom}_{R}(B,-)$ to 
 $B\stackrel{\pi}{\rightarrowtail} B\twoheadrightarrow B/\pi B$ we obtain the exact sequence
$$
\xymatrix{ 0 \ar[r] &  X\ar[r]^{\pi} & X
	\ar[r] & \mathrm{Hom}_{R}(B,B/\pi B)	\ar[r] & 0},
$$
or equivalently
$$
\xymatrix{ 0 \ar[r] &  X\ar[r]^{\pi} & X
	\ar[r] & \mathrm{Hom}_{\overline{K}}(V,V)	\ar[r] & 0},
$$
where $V=B/\pi B$. Therefore, we may  identify $X/\pi X$ with $\mathrm{Hom}_{\overline{K}}(V,V)$; we obtain by Remark \ref{Twin integers in the torsion-free case} that $\mathrm{Hom}_{\overline{K}}(V,V)$ is a free $\overline{K}G$-module, and so the same conclusion holds for $V$ by Lemma \ref{Modules over F_p[X]}.   It follows then by Remark \ref{Twin integers in the torsion-free case} that $B$ is CT, which concludes the proof. 
\end{proof}

\section*{Acknowledgment}
The first author is grateful to Prof. Ergun Yalcin for his interest and encouragement.


\end{document}